\documentclass{amsart}
\usepackage{amssymb}
\newcommand{\N}{\mathbb{N}}
\newcommand{\Z}{\mathbb{Z}}
\newcommand{\C}{\mathbb{C}}
\newcommand{\R}{\mathbb{R}}
\renewcommand{\P}{\mathcal{P}}
\newcommand{\X}{\mathcal{X}}
\newcommand{\res}{\mathrm{res}}
\newtheorem{Theo}{Theorem}
\newtheorem{Lem}{Lemma}
\newtheorem{cor}{Corollary}
\begin{document}
\title{Natural boundaries of Dirichlet series}
\author[G. Bhowmik]{Gautami Bhowmik}
\author[J.-C. Schlage-Puchta]{Jan-Christoph Schlage-Puchta}
\begin{abstract}
We prove some conditions on the existence of natural boundaries of  Dirichlet
series.  We show that  generically the presumed boundary is the natural one.
We also give an application of natural boundaries in determining asymptotic results.
\end{abstract}
\maketitle
\section{Introduction}
 It is very difficult to say much about the meromorphic
continuation of  Euler products of Dirichlet series beyond the region of convergence.
The only general method to show the existence of a natural boundary is to
prove that every point of the presumed boundary is the limit point of either
poles of zeros of the function in question. In general, it does not suffice to
prove that each point is a limit point of poles or zeros of the single factors,
since poles and zeros might cancel. 

There are, of course, many examples of special cases where precise  information was
obtained, as was done by Estermann  \cite{Est} who proved that if 
there is an Euler product $D(s)=\prod _p h(p^{-s})$ where  $ h(Y)$ is a
ganzwertige polynomial,
then $D(s)$ is absolutely convergent for $\Re(s) >1$ and can be meromorphically continued 
to the half plane $\Re(s) >0$ . 
If $h(Y)$ is a product of cyclotomic polynomials, then and only then can $D(s)$
be continued to the whole complex plane.
Dahlquist \cite{Dahl} generalised this result to $h$ being any analytic function with isolated singularities
within the unit circle.  This line of investigation was continued to several variables. 
 Analytic continuations of  multiple zeta functions 
received a lot of attention in recent years, especially by the Japanese
school. The known examples confirm the 
 belief that if there is an obvious candidate for a boundary,
it is the natural boundary.

However, in cases like  $D(s)=\prod _p h(p,p^{-s})$, which occur in the study
of algebraic groups (see, for example, \cite{duSG}), the above belief is yet to
be confirmed. Thus 
a simple case like $D(s) = \prod_p \Big(1+p^{-s} + p^{1-2s}\Big)$ \cite{duSPre}
remains resistant to our understanding.
In this paper, we find some conditions such that too much 
cancellation among potential singularities becomes impossible and a
natural boundary is obtained (Theorem 1). This helps us give partial answers to
series like the one just cited  (Corollary 1).

Our belief in the possibility of meromorphic continuation upto an expected
domain
is strengethened by a generic consideration. Several authors (see, for example,
 \cite{Ka} or \cite{Queff}) studied  Dirichlet series with
random coefficients and showed that such series almost surely have natural
boundaries.
In this paper, the definition of a random series is adapted to serve our
interest
and we prove that almost surely the  series thus defined has meromorphic continuation
upto the presumed half-plane (Theorem 3).

Finally, we show that the existence of a natural boundary can help in
obtaining  $\Omega$-results for  Dirichlet series associated to counting functions.
We prove that if $D(s)=\Sigma a(n)n^{-s} $ has a natural boundary at $\Re s=\sigma$,  then there does
not exist an explicit formula of the form 
$A(x) := \sum_{n\leq x} a_n = \sum_{\rho} c_\rho x^\rho +
\mathcal{O}(x^\sigma)$, where $\rho$ is a zero of the Riemann zeta function
and hence it is possible to obtain a term 
 $\Omega(x^{\sigma -\epsilon})$ in the asymptotic expression for $A(x)$. We
 treat the Igusa-Weil function for algebraic groups for rank 3 (interpreted as
 a counting function) as an example of this manoeuvre.

\section{Criteria for a boundary}
Combinatorics of sets of integers and real numbers are often an
ingredient of the proofs of natural boundary; confer, for instance,
Dahlquist's concept of vertex numbers \cite{Dahl} . The following
Lemma shows that in an appropriate setting, there cannot be too much
cancellations among potential singularities. For a discrete set
$\X\subset[0, \infty)$, denote by $\X(x)$ the number of elements of
$\X\cap[0, x]$. 
\begin{Lem}
\label{Lem:NumComb}
Let $0<\epsilon<1$ be given. Let $\P$ be a set of primes with
$\P((1+\epsilon)x)-\P(x)\gg x^\theta\log^2 x$ and let $\X\subset[0, T]$ be a
discrete set of real numbers satisfying $\X(x)\ll x\log x$ and
$T\in\X$. Assume that for all $p\in\P$ and all $x\in\X$ with
$\frac{x}{p}\in[1, 1+\epsilon]$, there exists some $n\in\N\setminus\P$ and
some $y\in\X$ with $\frac{x}{p}=\frac{y}{n}$. Then we have
$\theta<\frac{\sqrt{5}-1}{2}=0.618\ldots$.
\end{Lem}
Of course, the condition on the growth of $\X$ is somewhat arbitrary; the
formulation chosen here being  dictated by the application, where $\X$
will be chosen as a set of abscissae of certain zeros of $\zeta$.
\begin{proof}
For each $p\in\P\cap[\frac{T}{1+\epsilon}, T]$, there exists some $y_p\in\X$
such that there is some $n\not\in\P$ with $\frac{T}{p}=\frac{y_p}{n}$. For
each such $p$, choose an integer $a_p$ with the property that $a_py_p\in\X$,
but $ka_py_p\not\in\X$ for all integers $k\geq 2$. Next, for each $p$ choose
prime numbers $p'\in\X\cap[\frac{a_py_p}{1+\epsilon}, a_py_p]$ with $p'\nmid
a_p$. For each such choice, there exists an element $z_{p, p'}\in\X$ and an
integer $n'\not\in\P$, such that $\frac{a_py_p}{p'}=\frac{z_{p, p'}}{n'}$. We
claim 
that $z_{p, p'} = z_{\tilde{p}, \tilde{p}'}$ implies that $\{p, p'\} =
\{\tilde{p}, \tilde{p}'\}$. In fact, we have
\[
z_{p, p'} = z_{\tilde{p}, \tilde{p}'}\quad\Leftrightarrow T\frac{a_pnn'}{pp'}
= T\frac{a_{\tilde{p}}\tilde{n}\tilde{n}'}{\tilde{p}\tilde{p}'}.
\]
By construction, all the integers $a_p, a_{p'}, n, \tilde{n} ,n', \tilde{n}'$
are at most $T$, while $p$ and $\tilde{p}$ are at least
$\frac{T}{1+\epsilon}$. Hence, neither $p$ nor $\tilde{p}$ cancel, and we
either obtain $p=\tilde{p}$, or $p=\tilde{p}'$ and $p'=\tilde{p}$. In the
latter case the sets $\{p, p'\}$ and $\{\tilde{p}, \tilde{p}'\}$ coincide, and
we are done. To deal with the first possibility, note that the choice of $a_p$
and $n$ depend  only on $p$, hence, $z_{p, p'} = z_{\tilde{p}, \tilde{p}'}$
implies $\frac{n'}{p'} = \frac{\tilde{n}'}{\tilde{p}'}$. If $p'$ does not
divide $n'$, this implies $p'=\tilde{p}'$, and we obtain $\{p, p'\} =
\{\tilde{p}, \tilde{p}'\}$ as in the first case. Finally, we show that by 
construction of $p'$ and $n'$, $p'$ can never divide $n'$. In fact, $p'\neq
n'$, since otherwise $n'$ would be in $\P$, contrary to our choice of $n'$. Thus,
$\frac{n'}{p'}=k$ would be an integer $\geq 2$, and we would obtain $z_{p,
  p'}=ka_py_p$, which would contradict our definition of $a_p$. Hence, we have
shown that $z_{p, p'}$ indeed determines the set $\{p, p'\}$. Next, we
estimate the number of sets $\{p, p'\}$ in the above manner. By assumption,
there are $\gg T^\theta\log^2 T$ choices for $p$. The growth condition for
$\X$ implies that there are at least $T^\theta$ prime numbers $p$, such that
$y_p>T^\theta$. For each such prime $p$, the number of choices for $p'$ is
$\gg y_p^\theta\log^2 y_p\gg T^{\theta^2}\log^2 T$, hence, the total number of
pairs $(p, p')$ is of order of magnitude $T^{\theta+\theta^2}\log^2 T$, and
the number of unordered sets differs from this quantity by a factor of 2 at
most. Hence, we obtain the estimates
\[
T^{\theta+\theta^2}\log^2 T\ll |\{z_{p, p'}\}| \leq |\X |\ll T\log T,
\]
which implies $\theta+\theta^2<1$, that is, $\theta<\frac{\sqrt{5}-1}{2}$.
\end{proof}
\begin{Theo}
\label{Thm:SpecialBoundary}
Assume the Riemann $\zeta$-function has infinitely many zeros off the line
$\frac{1}{2}+it$. Suppose that $f$ is a function of the form
$f(s)=\prod_{\nu\geq1}\zeta(\nu(s-\frac{1}{2})+\frac{1}{2})^{n_\nu}$ where the
exponents $n_\nu$ are rational integers and the series
$\sum\frac{n_\nu}{2^{\epsilon\nu}}$ converges absolutely for every
$\epsilon>0$. Then $f$ is holomorphic in the 
half plane $\Re s>1$ and has meromorphic continuation in the half plane $\Re
s>\frac{1}{2}$. Denote by $\P$ the set of prime numbers $p$, such that
$n_p>0$, and suppose that for all $\epsilon>0$ we have
$\P((1+\epsilon)x)-\P(x)\gg x^\frac{\sqrt{5}-1}{2}\log^2 x$. Then the line $\Im
s=\frac{1}{2}$ is the natural boundary of $f$; more precisely, every point of
this line is accumulation point of zeros of $f$.
\end{Theo}
\begin{proof}
Let $\epsilon>0$ be given. Then only finitely many factors in the infinite
product have a pole in the half-plane $\Re s>\frac{1}{2}+\epsilon$, and we
have $\zeta(\nu(s-\frac{1}{2})+\frac{1}{2})-1 \sim
2^{\nu(s-\frac{1}{2})+\frac{1}{2}}$ for $\nu\rightarrow\infty$, thus apart
from a discrete subset we have
\begin{eqnarray*}
f(s) & = & \exp\big(\sum_\nu
n_\nu\log\zeta(\nu(s-\frac{1}{2})+\frac{1}{2})\big)\\
 & = & \exp\big(\sum_\nu
\frac{n_\nu}{2^{\nu(s-\frac{1}{2})+\frac{1}{2}}}+\mathcal{O}(1)\big),
\end{eqnarray*}
and by assumption this sum converges absolutely for all $s$ with $\Re s>\frac{1}{2}$ 
this sum converges absolutely for all $s$ with $\Re s>\frac{1}{2}$, 
hence, apart from a discrete set of poles, $f$ can be
holomorphically continued to the half-plane $\Re s>\frac{1}{2}$. 
We shall now prove that every point of the line $1/2+it$ is an accumulation 
point
of zeros or poles of $f$. To do so, note first that every point on this
line is accumulation point of zeros with real part $>1/2$ of factors
in the infinite product defining $f$. In fact, by assumption there are 
infinitely many zeros of $\zeta$ to the right of the line $\Re s= 1/2$,
thus, for every $\epsilon>0$ and every $t$ there is a zero $\rho=\sigma+iT$ of
$\zeta$, such that $\mathcal{P}(T/t)-\mathcal{P}(T/((1+\epsilon)t))\gg
(T/t)^\theta\log^2 (T/t)$, where $\theta=\frac{\sqrt 5-1}2$. In
particular, there exists a prime number $p$ with 
$n_p> 0$, such that $T/p\in[t, (1+\epsilon)t]$. Hence, to prove our claim,
we have to show that this zero cannot be cancelled by 
poles stemming from other factors. We cannot do so for a single point,
however, using Lemma~\ref{Lem:NumComb},
we can show that not all such poles or zeros can be cancelled.
In fact, let $\X$ be the set imaginary parts of zeros of $\zeta$
lying on the line passing through $\frac{1}{2}$ and $\rho$ and having
positive imaginary part. Let $T$ be the maximum of $\X$, that is
$\X\subset[0, T]$. Since the number of all zeros of $\zeta$
with imaginary part $\leq x$ is of magnitude $\mathcal{O}(x\log x)$,
we have a fortiori $\X(x)\ll(x\log x)$. If
$\frac{\rho-1/2}{p}+\frac{1}{2}$ is not a zero of $f$, there has to be
some integer $\nu$ and a zero $\rho'$ of $\zeta$, such that $n_\nu<0$,
and $\frac{\rho-1/2}{p}=\frac{\rho'-1/2}{\nu}$, that is, $\rho'$ is on
the line through $\frac{1}{2}$ and $\rho$, and has positive imaginary
part, thus, $\Im\;\rho'\in\X$. Moreover, for every $p\in\mathcal{P}$ we
have $n_p>0$, whereas $n_\nu<0$, thus, $\nu\not\in\mathcal{P}$. Since
we are not restricted in our choice of $p$ and $\rho$ except for the
conditions $p\in\mathcal{P}$ and $\frac{\Im\;\rho}{p}\in[t,
(1+\epsilon)t]$, we find that we can apply Lemma~\ref{Lem:NumComb} to
deduce $\theta<\frac{\sqrt{5}-1}{2}$. However, this contradicts our
assumption on the density of $\mathcal{P}$, which show that there is
some $p\in\mathcal{P}$ and a zero $\rho$ of $\zeta$, such that
$\frac{\rho-1/2}{p}+\frac{1}{2}$ is a zero of $f$, that is, in every
square of the form $\{s:\Re\;s\in[\frac{1}{2}, \frac{1}{2}+\epsilon], 
\Im\;s\in[t, t+\epsilon]\}$, there is a zero of $f$, that is, every
point of the line $\Re\;s=\frac{1}{2}$ is accumulation point of zeros
of $f$, and since $f$ is not identically zero, this line forms a
natural boundary.
\end{proof}
We can use Theorem~\ref{Thm:SpecialBoundary} to give a partial solution to an
``embarrassingly innocuous looking'' case (see \cite[sec. 3.2.4]{duSPre}).
\begin{cor}
\label{cor:innocent1}
Suppose that there are infinitely many zeros of $\zeta$ off the line
$\frac{1}{2}+it$. Then the function
\[
f(s) = \prod_p \Big(1+p^{-s} + p^{1-2s}\Big)
\]
has meromorphic continuation to the half plane $\Re s>\frac{1}{2}$, and the
line $\Re s=\frac{1}{2}$ is the natural boundary of $f$.
\end{cor}
\begin{proof}
The function $f$ can be expanded into a product of $\zeta$-functions as
follows:
\begin{eqnarray*}
f(s) & = & \frac{\zeta(s)\zeta(2s-1)\zeta(3s-1)}{\zeta(2s)\zeta(4s-2)} R(s)\\
&&\times\;\prod_{m\geq1} \frac{\zeta((4m+1)s-2m)}
{\zeta((4m+3)s-2m-1)\zeta((8m+2)s-4m)},
\end{eqnarray*}
where $R(s)$ is a function holomorphic in some half-plane strictly larger than
the half-plane $\Re s>\frac{1}{2}$. Denote by $D$ the infinite product on the
right of the last equation. Then we have 
\[
D(s) = \prod_{m\geq1} \frac{\zeta((4m+1)s-2m)}
{\zeta((4m+3)s-2m-1)}\prod_{m\geq1} \zeta((8m+2)s-4m)^{-1} = \prod\nolimits_1
\times \prod\nolimits_2,
\]
say. $\prod_1$ is of the form considered in Theorem~\ref{Thm:SpecialBoundary}, 
whereas for $\Re s>\frac{1}{2}$, $\prod_2$ is an absolutely
convergent product of values of $\zeta$ in the half-plane $\Re s>1$, thus,
$\prod_2$ is holomorphic and non-vanishing in the half-plane $\Re
s>\frac{1}{2}$, and therefore cannot interfere with zeros of $\prod_1$. Hence,
every point of the line $\Re s=\frac{1}{2}$ is an accumulation point of zeros of
$D$, and $D$ cannot be continued meromorphically beyond this line.
\end{proof}

Another application is the following, which partially resolves a class of
polynomials considered in \cite[Theorem~3.23]{duSPre}.

\begin{Theo}
Let $D(s)=\prod W(p, p^{-s})=\prod\zeta(ms-n)^{c_{nm}}$ be a
Dirichlet-series, such that all local zeros are to the left of the line
$\Re s=\beta$, where $\beta$ is the largest limit point of the set
$\{\frac{n}{m}: c_{nm}\neq 0\}$. Suppose that the number $P(x)$ of prime
numbers $p$ such that there is some $n$ with $c_{np}\neq 0$ and $n/p+1/2p >
\beta$ satisfies $P((1+\epsilon)x)-P(x) \gg x^{\frac{\sqrt 5-1}2}\log^2 x$. Then $\beta$ is
the natural boundary for $D$. 
\end{Theo}
\begin{proof}
 For any
$\epsilon>0$, there is some $N$, such that in the half-plane $\Re
s>\beta+\epsilon$ the product
$\prod_{n<N}\zeta(ms-n)^{c_{nm}}$ has the same zeros and poles as
$D(s)$. Hence, to prove that the line $\Re s=\beta$ is the natural boundary of
$D(s)$ it suffices to show that for every fixed $t_0\in\R$ and $\delta>0$
there is some $\epsilon>0$ such that for $N$ sufficiently large the product
$\prod_{n<N}\zeta(ms-n)^{c_{nm}}$ has a pole or a zero in the rectangle $R$
defined by the conditions 
$\beta+\epsilon < \Re s<\beta+\delta$, $t_0<\Im s <t_0+\delta$. The latter
would follow, if we could show that there exist integers $n, m$ with
$c_{nm}\neq 0$ and a zero $\rho$ of $\zeta$, such that $\frac{n+\rho}{m}\in
R$, and such that for all other values $n'$ and $m'$,
$\rho'=m'\frac{n+\rho}{m}-n'$ is not a zero of $\zeta$. Suppose first that
$\zeta$ has infinitely many zeros off the line $\Re s=\frac{1}{2}$. Then we
choose one such zero $\rho_0$ with sufficiently large imaginary part, and apply
Lemma~\ref{Lem:NumComb} with $\mathcal{P}$ being the set of primes 
$p$ such that there is some $n$ with $c_{np}\neq 0$ and $n/p+1/2p > \beta$,
and $\mathcal{X}$ being the set of all imaginary parts of roots of $\zeta$
of the form $m'\frac{n+\rho_0}{m}-n'$ to obtain a contradiction as in the
proof of Theorem~\ref{Thm:SpecialBoundary}.

Now suppose that up to a finite number of counterexamples, the Riemann
hypothesis holds true. Since these finitely many zeros off the line of $\zeta$
can only induce a discrete set of zeros of $D(s)$ apart from a possible
accumulation points on the real line, we can totally neglect these
zeros. Similarly, we forget about all pairs $n, m$ apart from those that can
induce zeros to the right from $\beta$; in particular, we may assume that
$\beta$ is the only limit point of the set of all occurring
fractions $\frac{n}{m}$. Finally, we can neglect finitely many pairs $n, m$ and
assume that all fractions $\frac{n}{m}$ are in an arbitrarily small interval
around $\beta$. The
contribution of a zero $\rho$ induced by some $c_{nm}\neq 0$ can be cancelled
by a zero $\rho'$ only if there are integers $n', m'$ with
$m(\frac{1}{2}+i\gamma)-n=m'(\frac{1}{2}+i\gamma')-n'$, that is,
$m\gamma=m'\gamma'$, and $m-2n=m'-2n'$. Without loss we may assume that
$\beta\neq\frac{1}{2}$, that is, $\frac{n}{m}-\frac{1}{2}$ is bounded away
from 0. Then the second equation implies an upper bound for $m'$, that is, for
at each cancellation among zeros there are only finitely many zeros concerned,
that is, we may assume that among these $\rho$ is the one with largest
imaginary part. But now we can apply Lemma~\ref{Lem:NumComb} again, this time
to the set of all zeros of $\zeta$, and obtain again a contradiction.
\end{proof}

\section{A random series}

Although the problem to decide whether a given Dirichlet-series can be
meromorphically extended to the whole complex plane may be very
difficult, we believe that in most cases the obvious candidate of a natural
boundary is in fact the natural boundary. This belief is strengthened
by the following theorem, which shows that this conjecture is
generically true. Note that our definition of a random series differs
from the usual one, in which random coefficients are used (for example
in Kahane \cite{Ka} or Qu\'effelec \cite {Queff}). 
 The following definition appears to be better suited.

\begin{Theo}
\label{thm:RandomCont}
Let $(a_\nu), (b_\nu), (c_\nu)$ be real sequences, such that $a_\nu,
b_\nu\to\infty$, and set $\sigma_h=\limsup\limits_{\nu\to\infty}
-\frac{b_\nu}{a_\nu}$. Let $\epsilon_\nu$ be a sequence of independent
real random 
variables, such that
\[
\liminf_{\nu\to\infty} \max_{x\in\R} P(\epsilon_\nu=x) = 0,
\]
and suppose that for $\sigma>\sigma_h$ the series
\begin{equation}
\label{eq:ConvCond}
\sum_{\nu=1}^\infty \frac{|c_\nu+\epsilon_\nu|}{2^{a_\nu\sigma+b_\nu}}
\end{equation}
converges almost surely. Then with probability 1 the function
\[
Z(s) = \prod\limits_{\nu=1}^\infty \zeta(a_\nu s+b_\nu)^{c_\nu+\epsilon_\nu}
\]
is holomorphic in the half-plane $\Re\;s>\sigma_h$ and has the line
$\Re\;s=\sigma_h$ as its natural boundary.
\end{Theo}
\begin{proof}
If the series (\ref{eq:ConvCond}) converges, then $Z$
can be written as a finite product of $\zeta$-functions multiplied by some
function which converges uniformly in the half-plane $\Re s>\sigma_h+\epsilon$
for each $\epsilon>0$.  Let $s_0=\sigma_h+it$ be a point on the supposed
boundary with $t\neq 0$ rational, and consider for a natural number $n$ the
square $S$
with side length $\frac{2}{n}$ centred in $s_0$, that is, the set
$[\sigma_h-\frac{1}{n}, \sigma_h+\frac{1}{n}]\times [t-\frac{1}{n},
t+\frac{1}{n}]$. Let $\epsilon>0$ be given. We show that with probability
$>1-\epsilon$ the function $Z$ is not meromorphic on $S$, or has a zero or a
pole in $S$. Once we have shown
this, we are done, for if $s_0$ were an interior point of the domain of
holomorphy of $D$, there would be some $n$ such that $Z$ would be holomorphic on $S$, and
have a zero or a pole in $S$ almost surely. Letting $n$ tend to $\infty$, we
see that $s_0$ is either a pole or a zero, or a cluster point of poles or
zeros. Hence, with probability 1, every point with rational imaginary part on
the line $\Re s=\sigma_h$ is a pole, a zero, or a cluster point of poles or
zeros. Hence, $\sigma_h$ is a natural boundary of $Z$ almost surely.

To prove the existence of a pole or zero in $S$, note first that by
the same argument used to prove alsmost sure convergence to the right
of $\sigma_h$, we see that if for some
$\epsilon>0$ there are infinitely many indices $\nu$ with
$-\frac{b_\nu}{a_\nu}<\sigma_h-\frac{1}{n}$, the product defining $Z$
extended over all such indices converges uniformly in
$\Re\;s>\sigma_h-\frac{1}{n}$, hence, deleting these indices does not
alter our claim. In particular, we may assume that for all $\mu$
sufficiently large we have
$|\sigma_h-\frac{a_\mu}{b_\mu}|<\frac{1}{n}$, $a_\mu>3n$, $|a_\mu
t|>1000$, as well as $\max_{x\in\R} P(\epsilon_\mu=x)<\epsilon$. For
such an index $\mu$ set
\[
Z_\mu(s) = \prod\limits_{\nu\neq\mu}^\infty \zeta(a_\nu
s+b_\nu)^{c_\nu+\epsilon_\nu}.
\]
If $Z$ is meromorphic on $S$, so is $Z_\mu$. Let $D_1$ be the divisor of the
restriction of $Z_\mu$ to $S$, and let $D_2$ be the divisor of $\zeta(a_\mu
s+b_\mu)$ restricted to $S$. We have to show that
$D_1+(c_\mu+\epsilon_\mu)D_2$ is non-trivial with probability
$>1-\epsilon$. To do so, it suffices to show that $D_2$ is non-trivial, since
then $D_1+ x D_2$ is trivial for at most one value of $x$, and we assumed that
$\epsilon_\mu$ is not concentrated on a single value.
The preimage of $S$ under the linear map $s\mapsto a_\mu s+b_\mu$ is a square
of side $\ell>6$ and centre with real part of absolute value
$\leq\ell$ and imaginary part of absolute value $>1000$. Hence, the number
of zeros of $\zeta(a_\mu s+b_\mu)$ in $S$ equals $N(T+h)-N(T)$, where $N$
denotes the number of zeros of $\zeta$ with imaginary part $\leq T$, and $T$
and $h$ are certain real numbers satisfying $T\geq 1000$ and $h\geq
6$. Now Backlund \cite{Ba} showed that for $T>1000$ we have
\[
\left|N(T)-\frac{T}{2\pi}\log\frac{T}{2\pi}\right| \leq 0.7\log T,
\]
that is, $N(T+6)>N(T)$ for $T>1000$, which shows that $D_2$ is non-trivial, and
proves our theorem.
\end{proof}

\section{Natural boundaries and asymptotic formulae}

The hunt for natural boundaries has certainly some intrinsic interest,
however, in this section we show that the existence of a natural
boundary implies the non-existence of an asymptotic formula of a
certain kind. This leads to a lesser known kind of $\Omega$-result :
usually when proving an $\Omega$-result, one first derives an explicit
formula with oscillating terms and then shows that these terms cannot
cancel each other out for all choices of the parameters. Here we show
that even if we allow for infinite oscillatory sums to be part of the
main terms, we still get lower bounds for the error terms.

\begin{Theo}
\label{thm:ExpCont}
Let $a_n$ be a sequence of complex numbers, and suppose that there exist an
explicit formula of the form
\begin{equation}
\label{eq:explicite}
A(x) := \sum_{n\leq x} a_n = \sum_{\rho\in\mathcal{R}} c_\rho x^\rho +
\mathcal{O}(x^\theta), 
\end{equation}
where for some constant $c$ we have $|c_\rho|\ll(1+|\rho|)^c$ and
$|\mathcal{R}\cap\{s:\Re s>\theta,|\Im s|<T\}|\ll T^c$. Then the
Dirichlet-series $D(s)=\sum a_n n^{-s}$ can be meromorphically
continued to the half-plane $\Re s>\theta$. 
\end{Theo}
Since the condition $\Re s>\theta$ describes an open set, we could
have formulated this theorem with an error term 
$\mathcal{O}(x^{\theta+\epsilon})$ for every $\epsilon>0$, or with
$\mathcal{O}(x^{\theta-\epsilon})$ for some $\epsilon>0$ without
affecting the conclusion. We shall move freely between these different
formulations without further mention. 
\begin{proof}
Our claim does not change if we absorb finitely many of the summands
$c_\rho x^\rho$ into the sequence $a_n$. Thus we can assume that all $\rho$
satisfy $|\Im\;\rho|\geq 1$.

Set $A_0(x)=A(x)$, $A_{k+1}(x)=\sum_{\nu\leq x} A_k(\nu)$. Then there exists an
explicit formula
\[
A_k(x) = x^k \sum_{\rho\in\mathcal{R}_k} c_\rho^{(k)} x^\rho +
\mathcal{O}(x^\theta),
\]
where $\mathcal{R}_k$ is contained in the set of all numbers of the form
$\{\rho-j:\rho\in\mathcal{R}, j\in\N\}$, and $c_\rho^{(k+1)} =
\frac{c_\rho^{(k)}}{\rho} +
\mathcal{O}\big(\max\limits_{j\in\N}c_{\rho+j}^{(k)}\big)$. By induction on $k$
we obtain
\[
c_\rho^{(k)}\ll\max\{|c_{\rho+j}|: j\in\N\} \rho^{-k+\max\{j:\rho+j\in\mathcal{R}\}},
\]
where $c_{\rho+j}$ is understood to be 0, if $\rho+j\not\in\mathcal{R}$.
Combining this estimate with the assumption on the number of elements
in $\mathcal{R}$, we see that there exists some $k$ such that the
explicit formula for $A_k$ 
converges absolutely. Note that we can immediately delete all terms
with $\Re\rho<\theta$, and $\Re\rho$ is bounded, since otherwise the explicit
formula for $A(x)$ would not converge in any sense. Thus, putting
$M=\lceil\sup\{\Re\;\rho:\rho\in\mathcal{R}\}-\theta\rceil$, we obtain
\[
c_\rho^{(k)}\ll\max\{|c_{\rho+j}|: j\in\N\} \rho^{-k+M}.
\]
Applying partial summation and interchanging the order
of summations, which is now allowed since the explicit formula is absolutely
converging, we find
\[
D(s) = \sum_{n\geq 1}\sum_{\rho\in\mathcal{R}_k}c_\rho^{(k)}
n^{k+\rho}\Delta^{k+1} n^{-s} + R(s),
 \]
where $R(s)$ is holomorphic in $\Re s>\theta$, and $\Delta$ denotes the
difference operator. Using Laurent expansion, we have for every $N$
the asymptotic formula 
\[
\Delta^{k+1} n^{-s} = \sum_{i=0}^N a_i(s) n^{-s-k-i-1} + \mathcal{O}(n^{-s-k-N-2})
\]
where the coefficients $a_i$ are polynomials of degree $i+k+1$. 
Inserting this expression in the previous formula, we obtain 
\[
D(s) = \sum_{n\geq 1}\sum_{\rho\in\mathcal{R}_k^*} c_\rho^{(k)}
n^{\rho-s}\Big( \sum_{i=0}^N 
a_i(s) n^{-i-1} + \mathcal{O}(n^{-N-1})\Big).
\]
Choosing $N$ sufficiently large, the error term yields a function
holomorphic in $\Re\;s>\theta$, and collecting all terms coming from one zero
$\rho$ which are independent of $n$ into one polynomial, we obtain
\begin{equation}
\label{eq:DRepresent}
D(s) = \sum_{\rho\in\mathcal{R}_k^*}P_\rho(s)  \zeta(s-\rho-1) + R^*(s),
\end{equation}
where $R^*(s)$ is holomorphic in $\Re\;s>\theta-1$, and $P_\rho$ is a polynomial
of degree $\leq N+k$ with coefficients $\ll c_\rho^{(k)}\ll|\rho|^{-k+M}$. 
We claim that this series is absolutely and uniformly converging in each
domain of the form $D=\{s:\Re s>\theta+\epsilon, |\Im s|<T\}$, apart from the
poles of 
$\zeta$ occurring explicitly. To prove this, we first have to estimate
$|P_\rho(s)|$. The bounds for the degree and the coefficients imply
\[
|P_\rho(s)|\ll C_{M, N} (1+|s|)^{N+k} |\rho|^{-k+M}.
\]
Since we only care about convergence, we may neglect finitely many
terms. Thus we restrict our considerations to zeros $\rho$ with
$|\Im\;\rho|>T^2$, that is, $|\rho|>|s|^2$. Finally, the functional equation for
$\zeta$ implies $\zeta(s)\ll(1+|\Im\;s|^{\max(\frac{1-\Re\;s}{2}, 0)+\epsilon})$, and we obtain
\begin{eqnarray*}
P_\rho(s)  \zeta(\rho-s) & \ll & (1+|\rho|)^{-k+M}(1+|\Im\;s|^{\max(\frac{\Re\;s-\rho+1}{2},
  0)+\epsilon}(1+|s|))^{N+k}\\
 & \ll & (1+|\rho|)^{-k+M+\frac{N+k}{2} + \max(\frac{\Re\;s-\rho+1}{4},  0)+\epsilon}\\
 & \ll & (1+|\rho|)^{-c-2},
\end{eqnarray*}
provided that
\[
k>4+2c+2M+2N+\max(\frac{\Re\;s-\rho+1}{2},0).
\]
Hence, the terms
belonging to $\rho$ are of order $\mathcal{O}\big((1+|\rho|)^{-c-2}\big)$,
whereas their number up to some constant $T$ is of order $\mathcal{O}(T^c)$,
hence, the series (\ref{eq:DRepresent}) converges absolutely and uniformly in
$D$. Hence, it represents a function holomorphic in $\Re s>\theta$, with the
exception of the discrete set of poles contained in $\mathcal{R}_k^*$. Since
for sufficiently large real part the right hand side of (\ref{eq:DRepresent})
represents $D(s)$, we deduce that this representation yields a
meromorphic continuation of $D$ to the half-plane $\Re s>\theta$.
\end{proof}
\begin{cor}
\label{Cor:NoExplicite}
Let $a_n$ be a sequence of complex numbers such that the generating
Dirichlet-series has a natural boundary at $\Re s=\sigma_h$. Then there does
not exist an explicit formula of the form (\ref{eq:explicite}). In
particular, for any sequence $\alpha_i, \beta_i$, $1\leq i\leq k$  and any
$\epsilon >0$ we have
\[
A(x) = \sum \alpha_i x^{\beta_i} + \Omega(x^{\sigma_h -\epsilon}).
\]
\end{cor}
In general, even if $D(s)$ is meromorphic in the entire plane we cannot expect
to obtain an explicit formula, since the integral taken over the shifted path
of integration need not converge. For example, for the Dirichlet-divisor
problem we have an $\Omega$-estimate of size $x^{1/4}$, whereas the
corresponding Dirichlet-series $\zeta^2(s)$ is meromorphic on $\C$. However,
we can obtain explicit formulae after attaching a sufficiently smooth weight
function. To do so, we need some bounds on the growth of the Dirichlet-series
in question.
\begin{Lem}
\label{Lem:Path}
Let $W\in\Z[X, Y]$ be a polynomial with $W(0, 0)=1$ and not containing the
monomial $X$. Let $D(s)=\prod_p W(p^{-1}, p^{-s})$ be the associated
Dirichlet-series, and let $\sigma_o$ be the abscissa of obvious meromorphic
continuation, and let $\sigma>\sigma_o$ be a real number.
\begin{enumerate}
\item There exists a $\mu(\sigma)$ such that $D(s)$ is the quotient of two
  functions $f_1, f_2$, both of which are holomorphic in the half-plane
  $\Re s\geq\sigma_o$ up to a bounded number of poles on the real axis, and
  satisfy $|f_i(\sigma+it)|\ll |t|^{\mu(\sigma)}$ for $|t|>1$.
\item The number of poles of $D$ in the domain $\Re s\geq\sigma$,
  $|\Im s|\leq T$ is bounded above by $c_\sigma T\log T$.
\item There is some $\mu^*(\sigma)$, such that for every $\epsilon>0$ and $T$
  sufficiently large there exists a path 
  $\gamma:[0, 1]\to\C$ 
  consisting of horizontal and vertical lines only, which is contained in the
  strip $\sigma\leq\Re s\leq\sigma+\epsilon$, has length $\leq (2+\epsilon
  T)$, such that $\Im\gamma(0)=-T$, $\Im\gamma(1)=T$, and
  $|D(s)|<e^{\mu^*(\sigma)\log^2 T}$ on all of $\gamma$.
\end{enumerate}
\end{Lem}
Note that the third statement is an adaptation of a result due to
Tur\'an \cite[Appendix G]{TuranAna}.
\begin{proof}
For each $\sigma>\sigma_o$, there exists a finite product of the form
$D^*(s)=\prod_{\kappa=1}^k \zeta(a_\kappa s+b_\kappa)^{c_\kappa}$, such that
$D(s)=D^*(s)R(s)$ with $R(s)$ holomorphic and bounded in the half-plane
$\Re s>\sigma$. Collecting terms with $c_\kappa>0$ in $f_1$, and terms with
$c_\kappa<0$ in $f_2$, the first statement follows from the fact that
$\zeta(s)$ grows only polynomially in each strip of bounded width. Moreover,
the number of poles of $D^*$ in the region $\Re s\geq\sigma$, $|\Im s|\leq
T$ is bounded above by some multiple of the number of zeros of $\zeta(s)$ in
the domain $\Re s>0$, $|\Im s|\leq T\max_k a_k$, which implies the second
assertion. For the third note that for each $s$ with $\Re s>\sigma$ we have
\[
\frac{{D^*}'}{D^*}(s) = \sum_\rho \frac{m_\rho}{s-\rho} + \mathcal{O}(\log T),
\]
where the sum runs over all poles and zeros of $D^*(s)$ with
$|\Im(\rho-s)|<$, and $m_\rho$ is the (signed) multiplicity
of the pole $\rho$. The same argument when used to prove the second assertion also
yields that for $|T|>2$ the number of poles and zeros $\rho$ of $D^*$ with
$T\leq\Im\rho\leq T+1$ is $\leq c_\sigma|T|$, hence, there is some
$\sigma'\in[\sigma, \sigma+\epsilon]$, such that there is no pole or zero
$\rho$ of $D$ with $T\leq\Im\rho\leq T+1$ and
$|\Re\rho-\sigma'|<\frac{\epsilon}{c_\sigma\log T}$. Hence, on this line
segment, we have $\big|\frac{{D^*}'}{D^*}\big|\ll\log^2 T$. Choosing $T$ in
such a way that $D^*$ has no poles or zeros in the half-strip $\Re s>\sigma$,
$|T-\Im s|<\frac{\epsilon}{c_\sigma\log T}$, we find that there exists a path
$\gamma$ as desired such that each point on $\gamma$ can be linked to a point
in the half plane of absolute convergence of $D$ by a path of length $\ll 1$,
such that $\big|\frac{{D^*}'}{D^*}\big|\ll\log^2 T$ on this path. Hence, we
deduce $D(s)<e^{\mu^*(\sigma)\log^2 T}$ on $\gamma$.
\end{proof}
Now we give an example. In \cite{BG} we found a bijection between right cosets 
of $2t\times 2t$ symplectic matrices
and submodules of finite index of $\Z^{2t}$ which are equal to their duals
and which we call polarised. The counting function obtained corresponds to the
$p$-adic
zeta function of Weil-Igusa and occurs, for example, in \cite{duSG}.

\begin{Theo}
\label{thm:Axexplicite}
Denote by $a_n$ the number of polarised submodules of
$\Z^6$ of order $n$. Then we have for every $\epsilon>0$ 
\begin{equation}
\label{eq:Axexplicite}
A(x) := \sum_{n\geq 1} a_n e^{-n/x} = c_1 x^{7/3} + c_2 x^2 + c_3 x^{5/3} +
\sum_\rho \alpha_\rho x^{\frac{\rho+8}{6}} + \mathcal{O}(x^{4/3+\epsilon}), 
\end{equation}
where $\rho$ runs over all zeros of $\zeta$, and the coefficients
$c_1$, $c_2$, $c_3$, and $\alpha_\rho$ are numerically computable
constants. More 
precisely, we have $c_1=2.830\ldots$, $c_2=1.168\ldots$, and
$c_3=0.1037\ldots$. Moreover, the error term cannot be improved to
$\mathcal{O}(x^{4/3-\epsilon})$ for any fixed
$\epsilon>0$.
\end{Theo}
\begin{proof}
The generating function for $a_n$ has the form \cite
 {BG} 
\begin{eqnarray*}
Z(s/3) & = & \zeta(s)\zeta(s-3)\zeta(s-5)\zeta(s-6)\prod_p
\Big(1+p^{1-s}+p^{2-s}+p^{3-s}+p^{4-s}+p^{5-2s}\Big)\\
 & = & \zeta(s)\zeta(s-3)\zeta(s-5)\zeta(s-6)\frac{\zeta(s-4)}{\zeta(2s-8)}\\
&&\qquad\times \prod_p
\Big(1+\frac{p^{1-s}+p^{2-s}+p^{3-s}+p^{5-2s}}{1+p^{4-s}}\Big)
\end{eqnarray*}
and in \cite{duSG} it was proved that $\Re\;s=\frac{4}{3}$ is the natural boundary
for the above.
The product over primes converges absolutely and uniformly in every
half-plane $\Re\;s>\frac{4}{3}+\epsilon$. Hence, $Z(s)$ has simple
poles at $7/3$, $2$ and $5/3$, poles at the zeros of $\zeta(6s-8)$,
and no other singularities in the half plane $\Re s>4/3$. Applying the
Mellin transform
\[
e^{-y} = \int\limits_{3-i\infty}^{3+i\infty} \Gamma(s)y^s\;ds
\]
we obtain
\[
A(x) = \frac{1}{2\pi i}\int\limits_{3-i\infty}^{3+i\infty} Z(s)\Gamma(s)x^s\;ds.
\]
For $\sigma$ and $\epsilon>0$ fixed, we have $\Gamma(\sigma+it)\ll
e^{-(\frac{\pi}{2}-\epsilon) t}$. We now choose a path as in
Lemma~\ref{Lem:Path}, and shift the integration to this path. Due to
the rapid decrease of $\Gamma$, we find that 
for $T=\log^3 x$ the integral on the new path is bounded above by
$x^{4/3+\epsilon}$. Hence, we obtain the formula
\[
A(x) = \sum_{\Re\rho>4/3+\epsilon} \Gamma(\rho)
x^\rho\res_{s=\rho} Z(s) + \mathcal{O}(x^{4/3+\epsilon}),
\]
where $\rho$ runs over $7/3$, $2$, $5/3$, and all complex numbers
$4/3+\rho/6$, where $\rho$ runs over all non-trivial zeros of $\zeta$.

To compute the values of $c_1, c_2, c_3$, we only have to compute 
the residuum of $Z(s)$ at these points, which does not pose any problems,
since the Euler products involved converge rather fast. We obtain the residue
$2.377, -1.168, 0.1149$, respectively, which yields the constants mentioned in
the Theorem. Using Mathematica, CPU-time for these computations was about 30
seconds. 

In view of \cite{duSG}, $Z(s)$ has a natural boundary on the
line $\Re s=4/3$, hence, the proof of the $\Omega$-result runs
parallel to the proof of Theorem~\ref{thm:ExpCont}.
\end{proof}

Having an explicit formula, we can use standard methods to draw
conclusions from it. For example, we have the following.

\begin{cor}
\label{cor:AxAsymp}
Define $A(x)$ as above. Then we have
\[
A(x) = c_1 x^{7/3} + c_2 x^2 + c_3 x^{5/3} + \mathcal{O}\big(x^{3/2}e^{-c\frac{\log
  x}{(\log\log x)^{2/3+\epsilon}}}\big)
\]
and
\[
A(x) = c_1 x^{7/3} + c_2 x^2 + c_3 x^{5/3} + \Omega_\pm
\big(x^{17/12-\epsilon}\big)
\]
\end{cor}
\begin{proof}
Note that apart from the poles at $7/3, 2, 5/3$ and $3/2$ all
singularities of $D(s)$ in the half-plane $\Re s>4/3$ come from zeros of
$\zeta(6s-8)$, hence, for a certain constant $c$ we have for all $\rho$
occurring in (\ref{eq:Axexplicite}) the relation $\Re\rho<\frac{3}{2} -
\frac{c}{(\log\log|\Im s|)^{2/3+\epsilon}}$. Since $\Gamma(s)$ decreases
exponentially fast on each line parallel to the imaginary axis, we see that
the contribution of a single zero is at most
\[
\max_{T>3} x^{\frac{3}{2} - \frac{c}{\log^{2/3+\epsilon} T}} e^{-c' T} \ll
x^{\frac{3}{2}} e^{-c\frac{\log x}{(\log\log x)^{2/3+\epsilon}}};
\]
moreover, the contribution of zeros with imaginary part $>\log^2 T$ is
negligible. Hence, the contribution of all zeros of $\zeta(6s-8)$ and the
error term in (\ref{eq:Axexplicite}) together give an error term of order
$x^{\frac{3}{2}} e^{-c\frac{\log x}{(\log\log x)^{2/3+\epsilon}}}$, and our
claim follows.

The $\Omega$-estimate follows from a standard application of Tur\'an's
theory of powersums, confer \cite[Chapter 47]{TuranAna}.
\end{proof}

Of course, these computations did not make use of the natural boundary
of $Z$, however, the existence of a natural boundary implies that there is a
limit to what can be achieved by complex analytic means.

\begin{tabular}{ll}
Gautami Bhowmik, & Jan-Christoph Schlage-Puchta,\\
Universit\'e de Lille 1, & Albert-Ludwigs-Universit\"at,\\
Laboratoire Paul Painlev\'e, & Mathematisches Institut,\\
U.M.R. CNRS 8524,  & Eckerstr. 1,\\
  59655 Villeneuve d'Ascq Cedex, & 79104 Freiburg,\\
  France & Germany\\
bhowmik@math.univ-lille1.fr & jcp@math.uni-freiburg.de
\end{tabular}
\end{document}